\pdfoutput=1
\documentclass[11pt]{amsart}
\usepackage{a4wide,amsfonts,graphicx,
amssymb,stmaryrd,amscd,amsmath,latexsym,amsbsy,amsthm}
\usepackage[all,knot,poly]{xy}
\usepackage[all,knot,poly]{xy}
\usepackage{color}
\usepackage{ifpdf}
\ifpdf 
  \usepackage{graphicx}
  \DeclareGraphicsExtensions{.pdf,.png,.mps}
  \usepackage{pgf}
  \usepackage{tikz}
\else 
  \usepackage{graphicx}
  \DeclareGraphicsExtensions{.eps,.bmp}
  \usepackage{pgf}
  \usepackage{tikz}
  \usepackage{pstricks}
\fi
\usepackage{epic}
\usepackage{wrapfig}
\let \ti=\widetilde
\let \ol=\overline

\newtheorem{thm}{Theorem}[section]

\newtheorem{cor}[thm]{Corollary}
\newtheorem{lem}[thm]{Lemma}
\newtheorem{prop}[thm]{Proposition}

\newcommand{\mc}{\mathcal}

\begin{document}
\title[On the tempered L-function conjecture]
{On the tempered L-function conjecture}

\author{Volker Heiermann}
\address{Clermont Universit\'e\\
Universit\'e Blaise-Pascal\\
Laboratoire de math\'ematiques\\
BP 10448\\
F-63000 Clermont-Ferrand\\
email: heiermann@math.univ-bpclermont.fr}
\author{Eric Opdam}
\address{Korteweg de Vries Institute for Mathematics\\
University of Amsterdam\\P.O. Box 94248\\1090 GE Amsterdam\\
The Netherlands\\ email: e.m.opdam@uva.nl}

\date{\today}
\thanks{The first named author has benefitted from help of the Agence
Nationale de la Recherche with reference ANR-08-BLAN-0259-02. The
second named author thanks the laboratoire de math\'ematiques of
the university Blaise-Pascal in Clermont-Ferrand for its
hospitality during the elaboration of this work.} \keywords{}
\begin{abstract}
We give a general proof of Shahidi's tempered $L$-function
conjecture, which has previously been known in all but one case.
One of the consequences is the standard module conjecture for
$p$-adic groups, which means that the Langlands quotient of a
standard module is generic if and only if the standard module is
irreducible and the inducing data generic. We have also included
the result that every generic tempered representation of a
$p$-adic group is a sub-representation of a representation
parabolically induced from a generic supercuspidal representation
with a non-negative real central character.
\end{abstract}
\maketitle
\section{Introduction}
Let $F$ be a non archimedean local field of characteristic $0$.
Let $G$ be the group of points of a quasi-split connected
reductive $F$-group.

By a parabolic subgroup (Borel subgroup, Levi subgroup, torus,
split torus) of G we will mean the group of points of an
$F$-parabolic subgroup ($F$-Borel subgroup, $F$-Levi subgroup,
$F$-torus, $F$-split torus) of the algebraic group underlying G.

Fix a Borel subgroup $B=TU$ of $G$, and let $T_0\subset T$ be the
maximal split torus in $T$. If $M$ is any semi-standard Levi
subgroup of $G$ (i.e. a Levi subgroup which contains $T_0$), a
standard parabolic subgroup of $M$ will be a parabolic subgroup of
$M$ which contains $B \cap M$.

Denote by $W$ the Weyl group of $G$ defined with respect to $T_0$
and by $w_0^G$ the longest element in $W$. By (\cite{Sh3}, section
3) we can fix a non degenerate character $\psi $ of $U$ which, for
every Levi subgroup $M$, is compatible with $w_0^Gw_0^M$. We will
still denote $\psi $ the restriction of $\psi $ to $M\cap U$.
Every generic representation $\pi $ of $M$ becomes generic with
respect to $\psi $ after changing the splitting in $U$.

Let $P=MU$ be a standard parabolic subgroup of $G$ and $T_M$ the
maximal split torus in the center of $M$. We will write $a_M^*$
for the dual of the real Lie-algebra $a_M$ of $T_M$, $a_{M,\mathbb
C}^*$ for its complexification and $a_M^{*+}$ for the positive
Weyl chamber in $a_M^*$ defined with respect to $P$. Following
\cite{W}, we define a map $H_M:M\rightarrow a_M$, such that
$\vert\chi (m)\vert _F=q^{-\langle\chi ,H_M(m)\rangle}$ for every
$F$-rational character $\chi\in a_M^*$ of $M$. If $\pi $ is a
smooth representation of $M$ and $\nu \in a_{M,\mathbb C}^*$, we
denote by $\pi _{\nu }$ the smooth representation of $M$ defined
by $\pi _{\nu }(m)=q^{-\langle\nu ,H_M(m)\rangle}\pi (m)$. (Remark
that, although the sign in the definition of $H_M$ has been
changed compared to the one due to Harish-Chandra, the meaning of
$\pi _{\nu }$ is unchanged.) The symbol $i_P^G$ will denote the
functor of parabolic induction normalized such that it sends
unitary representations to unitary representations, $G$ acting on
its space by right translations.

The parabolic subgroup of $G$ which is opposite to $P$ will be
denoted by $\overline{P}=M\overline{U}$.

Let $(\tau ,E)$ be an irreducible tempered $\psi $-generic
representation of $M$.

Put $\widetilde {w}=w_0^Gw_0^M$ and fix a representative $w$ of
$\widetilde {w}$ as in \cite{Sh3}. Then $w\overline{P}w^{-1}$ is a
standard parabolic subgroup of $G$. For any $\nu\in a_M^*$ there
is a Whittaker functional $\lambda _P(\nu, \tau ,\psi )$ on
$i_P^GV$. It is a linear functional on $i_P^GV$, which is
holomorphic in $\nu $, such that for all $v\in i_P^GV$ and all
$u\in U$ one has $\lambda _P (\nu, \tau ,\psi )((i_P^G\tau _{\nu
})(u)v)=\psi (u)\lambda _{P}(\nu,\tau ,\psi )(v)$. More precisely,
assuming that the space of $\tau $ is formed by Whittaker
functions, one can define $\lambda _P (\nu, \tau ,\psi )$ by (cf.
\cite{Sh1}, proposition 3.1) $$\lambda _P (\nu, \tau ,\psi
)(v)=\int _{U}(v(wu))(1)\ol{\psi (u)}du,$$ where $(v(wu))(1)$
denotes the value in $1$ of the Whittaker function $v(wu)$ in the
space of $\tau _{\nu }$. Remark that by Rodier's theorem
\cite{Ro}, $i_P^G\tau _{\nu }$ has a unique $\psi $-generic
irreducible sub-quotient.

For all $\nu $ in an open subset of $a_M^*$ we have an
intertwining operator $J_{\overline{P}\vert P}(\tau _{\nu
}):i_P^G\tau _{\nu }\rightarrow i_{\overline{P}}^G\tau _{\nu }$.
For $\nu $ in $(a_M^*)^+$ far away from the walls, it is defined
by a convergent integral $$(J_{\overline{P}\vert P}(\tau _{\nu
})v)(g)=\int _{\overline{U}} v(ug) du.$$ It is meromorphic in $\nu
$ and the map $J_{P\vert\overline{P}}J_{\overline{P}\vert P}$ is
scalar. Its inverse equals Harish-Chandra's $\mu $-function up to
a constant and will be denoted $\mu (\tau ,\nu )$.

Let $t(w)$ be the map $i_{\overline{P}}^GV\rightarrow
i_{w\overline{P}}^GwV$, which sends $v$ to $v(w^{-1}\cdot )$.
There is a complex number $C_{\psi }(\nu ,\tau ,w)$ \cite{Sh1}
such that $\lambda _P(\nu ,\tau ,\psi )=C_{\psi }(\nu ,\tau
,w)\lambda _{w\ol{P}}(w\nu ,w\tau ,\psi )t(w)$
$J_{\overline{P}\vert P}(\tau _{\nu })$. The function
$a_M^*\rightarrow\mathbb C$, $\nu\mapsto C_{\psi }(\nu ,\tau ,w)$
is meromorphic.

The local coefficient $C_{\psi }$ satisfies the equality $C_{\psi
}(\cdot ,\tau ,w)C_{\psi }(w(\cdot ),w\tau ,w^{-1})=\mu (\tau, \nu
)$ \cite{Sh1}.

In \cite{Sh3}, F. Shahidi attached to each irreducible component
$r_i$ of the adjoint action of the $L$-group $^LM$ of $M$ on
Lie$(\ ^LU)$, an $L$-function $L(s,\tau ,r_i)$, an $\epsilon
$-factor $\epsilon (s,\tau ,r_i,\psi )$, and a $\gamma $-factor
$\gamma (s,\tau ,r_i,\psi )$, such that $$\gamma (s,\tau ,r_i,\psi
)= \epsilon (s,\tau ,r_i,\psi )L(1-s,\tau ,r_i^{\vee })/L(s,\tau
,r_i).$$ In fact, $L(s,\tau ,r_i)$ equals the reciprocal of the
numerator of $\gamma (s,\tau ,r_i,\psi )$.

He showed that the local coefficient $C_{\psi }$ is equal to the
product of the factors $\gamma (is,\tau ,r_i,\psi )$ with a
holomorphic and non vanishing function (cf. \cite{Sh3}, identity
3.11).

The aim of this paper is to prove the following result:

\begin{thm} The local coefficient $\nu\mapsto C_{\psi }(\nu,\tau,
w)$ is holomorphic in the negative Weyl chamber, i.e. for $\nu\in
-(a_M^{*+})$, and the $L$-functions $L(s,\tau ,r_i)$ are
holomorphic for $s>0$.
\end{thm}

\null Remark that the holomorphicity of the local coefficient
$C_{\psi }$ is by the product formula for the local coefficient a
consequence of the holomorphicity of the $L$-functions, although
we will prove both parellel. The holomorphicity of the
$L$-function is known as Shahidi's tempered $L$-function
conjecture. It was originally stated in \cite{Sh3}, conjecture
7.1. It was later proved in all, but one case by different authors
(\cite{CSh}, \cite{MSh}, \cite{KH}, \cite{KW1}, \cite{KK},
\cite{KW2}). The remaining case concerned a group of type $E_8$
and its maximal Levi of type $E_6\times A_1$. If $\tau $ is
supercuspidal, the holomorphicity had already been shown in the
original paper of F. Shahidi \cite[proposition 7.3]{Sh3}.

As a corollary, one gets by \cite{HM} the following result, which
is called the standard modules conjecture:

\begin{cor} Let $\nu\in a_M^{*+}$. Denote by $J(\tau
,\nu )$ the Langlands quotient of the induced representation
$i_P^G\tau _{\nu }$. Then, the representation $J(\tau ,\nu )$ is
generic if and only if $i_P^G\tau _{\nu }$ is irreducible.
\end{cor}

    The paper is organized as follows: in section 2, we prove a
result which is not needed in the rest of the paper, but which
seems to us interesting in the context. It tells that any generic
irreducible tempered representations of $G$ is a
sub-representation parabolically induced by a supercuspidal
representation of a standard Levi subgroup with non negative
cental character.

    In section 3 the holomorphicity conjectures are reduced to
properties of functions, which can be defined in an affine Hecke
algebra context. The main ingredient here is the description of
the supercuspidal support of discrete series representations of
$p$-adic groups given in \cite{H2}. In section 4, we show that the
holomorphicity property for these functions holds under some
condition on the parameters which appear. We deduce this from the
unramified principal series case for split groups which is proved
in \cite{MSh}. In section 5, we finally prove that the parameters
coming from generic tempered representations of standard Levi
subgroups of $G$ satisfy this condition.

We thank F. Shahidi for some useful conversations and providing
the proof of lemma 6.1.

\section{An embedding property for generic discrete series}

The aim of this section is the proof of the proposition
\ref{prop:cuspidal support generic discrete series}. The proof has
been inspired by the paper \cite{Re}.

\begin{lem}\label{lem:Intertwining} Let $P=MU$ and $P_{\nu }=M_{\nu }U_{\nu }$ be two
standard parabolic subgroups of $G$, $P\subseteq P_{\nu }$. Let
$\sigma $ be a unitary $\psi $-generic supercuspidal
representation of $M$ and $\nu\in a_{M_{\nu }}^{*+}$. Write
$\ti{P_1}=w(P\cap M_{\nu })\ol{U_{\nu }}w^{-1}$ and
$\ti{P}=w\ol{P}w^{-1}$.

The intertwining operator $A_w= t(w^{-1})J_{\ol{\ti{P}} \vert
\ti{P_1}}(w\sigma _{\nu })$ is well defined and $\lambda _P(\nu
,\sigma,\psi )$ $A_w=c\lambda _{\ti{P_1}}(\ti{w}\nu ,w\sigma ,\psi
)$, where $c$ is a non zero constant.
\end{lem}

\begin{proof} The intertwining operator $A_w$ is well defined, because
any root $\alpha $ which is positive for $\ol{\ti{P}}$ and
negative for $\ti{P_1}$ verifies $\langle\ti{w}\nu ,\alpha ^{\vee
}\rangle >0$. One shows as in the case of opposite parabolic
subgroups that there is a meromorphic function $C_{\psi }(\nu
',w\sigma )$ depending on $\nu'\in a_{wM_{\nu }w^{-1}}^*$ such
that $\lambda _{\ti{P}_1}(\ti{w}\nu ',w\sigma ,\psi )=C_{\psi
}(\ti{w}\nu ',w\sigma )\lambda_{\ti{P}}(\nu ',\sigma ,\psi )$
$t(w^{-1})J_{\ol{\ti{P}} \vert\ti{P_1}}(w\sigma _{\nu '}).$ As the
intertwining operator depends effectively on a representation
induced from $M_{\nu }$ and $\ti{w}\nu $ is in the negative Weyl
chamber of $a_{wM_{\nu }w^{-1}}^*$ with respect to $w^{-1}M_{\nu
}w\ti{P}_1=\ol{P}_{\nu }$, it follows from the product formular
for the C-function and the fact that theorem {\bf 1.1} is known in
the supercuspidal case, that $C(\cdot ,w\sigma ,\psi )$ is
holomorphic in $\ti{w}\nu $. As in the supercuspidal case the
zeroes of the local coefficient $C_{\psi }$ lie on the unitary
axis, this proves the lemma.
\end{proof}

The following result is due to W. Casselman \cite{Ca}, proposition
4.1.4 and 4.1.6:

\begin{prop} Let $(\pi ,V)$ be an admissible representation of $G$,
$P_1=M_1U_1$ a semi-standard parabolic subgroup and $H$ an open
compact subgroup of Iwahori type with respect to $(P_1,M_1)$,
which means that $H=(H\cap U_1)(H\cap M)(H\cap \ol{U}_1)$.

Then there is an open compact subgroup $U_1'$ of $U_1$ such that
$V^H\cap V(U_1)\subseteq V(U_1')$. The spaces $(V^H)_a:=\pi
(1_{HaH})V$ with $a\in T_{M_1}$ positive for $P_1$ and such that
$aU_1'a^{-1}\subseteq H\cap U$ are all equal to the same space,
denoted $S_{P_1}^H(V)$. The Jacquet function $j_{P_1}^G$ induces
an isomorphism $S_{P_1}^H(V)\rightarrow (V)_{P_1}^{H\cap M_1}$.
\end{prop}

\begin{lem} (with the assumptions and notations of proposition
{\bf 2.2}) If $(\pi ', V')$ is a sub-representation of $(\pi ,V)$,
then one has $S_{P_1}^H(V)\cap V'=S_{P_1}^H(V')$.
\end{lem}

\begin{proof}
By definition, it is clear that $S_{P_1}^H(V')\subseteq
S_{P_1}^H(V)\cap V'$. On the other hand, if $v$ is an element of
$S_{P_1}^H(V)\cap V'$, then there is by proposition {\bf 2.2} an
element $v'$ in $S_{P_1}^H(V')$ such that
$j_{P_1}^Gv=j_{P_1}^Gv'$. As $S_{P_1}^H(V')\subseteq
S_{P_1}^H(V)$, it follows from proposition {\bf 2.2} that $v=v'$.
\end{proof}

\begin{lem} Let $P_1=MU_1$ be a semi-standard parabolic subgroup
with Levi factor $M$ and denote by $\ti{P}_1$ the semi-standard
parabolic subgroup which is conjugated by $w$ to $\ol{P}_1$. Let
$(\sigma ,E)$ be an admissible representation of $M$, let $H$ be
an open compact subgroup of $G$ of Iwahori type with respect to
$P_1$, such that there is a nonzero element $e$ in $E^{H\cap M}$.
Then there is a well defined element $v$ in $(i_{\ti{P}_1}^GwE)^H$
with support in $\ti{P}_1wH$ such that $v(w)=e$. It lies in
$S_{P_1}^H(i_{\ti{P_1}}^GwE)$.
\end{lem}

\begin{proof}
Choose an element $a\in T_M$ which satisfies the assumptions of
the proposition relative to $P_1$ and $i_{\ti{P_1}}^GwE$. One
observes that $\sigma (a^{-1})e$ lies in $E^{a^{-1}(H\cap M)a}$.
There is a well defined element $\ti{v}$ in $(i_{\ti{P}_1}^GwE)^H$
with support contained in $\ti{P}_1w(a^{-1}Ha)$ verifying
$\ti{v}(w)=\sigma (a^{-1})e$: this follows easily from the fact
that $a^{-1}Ha$ is also of Iwahori type relative to $P_1$ and
consequently $\ti{P}_1 w(a^{-1}Ha)=\ti{P}_1 wa^{-1}(H\cap U_1)a$.
A computation analog to the one in the proof of lemma {\bf 5.1} in
\cite{H1} gives then that $(i_{\ti{P}_1}^Gw\sigma
)(1_{HaH})\ti{v}$, multiplied by a convenient nonzero constant,
has the desired properties.
\end{proof}

\begin{prop}\label{prop:cuspidal support generic discrete series}
Let $\pi $ be a $\psi $-generic irreducible discrete
series representation of $G$. There exists a standard parabolic
subgroup $P=MU$ of $G$, a unitary $\psi $-generic supercuspidal
representation $(\sigma ,E)$ of $M$ and $\nu\in
\overline{a_M^{*+}}$, such that $\pi $ is a sub-representation of
$i_P^G\sigma_{\nu }$.
\end{prop}

\begin{proof}
It follows from results of \cite{Ro} that there exist $P=MU$,
$\sigma $ and $\nu $ as in the statement such that $\pi $ is a
sub-quotient of $i_P^G\sigma _{\nu }$.  In addition, $\pi $ is the
only irreducible $\psi $-generic sub-quotient of $i_P^G\sigma
_{\nu }$. From this one sees, that it is enough to show that there
is an irreducible sub-space of $i_P^G\sigma _{\nu }$, on which the
Whittaker functional $\lambda _P(\nu ,\sigma,\psi )$ does not
vanish.

Denote by $\Sigma (P)$ the set of reduced roots of $T_M$ in
$Lie(U)$, by $\Sigma _{\nu }$ the subset of roots $\alpha $ such
that $\langle\nu ,\alpha^{\vee }\rangle =0$ and by $M_{\nu }$ the
semi-standard Levi subgroup of $G$ containing $M$ obtained by
adjoining the roots in $\Sigma _{\nu }$ to $M$.

One has $\nu\in a_{M_{\nu }}^*$ and there is a parabolic subgroup
$P_{\nu }=M_{\nu }U_{\nu }$ such that $\nu $ lies in the positive
Weyl chamber of $a_{M_{\nu }}^*$ with respect to this parabolic
subgroup. The parabolic $P_{\nu }$ may not be standard, but
$P_{\nu }$ is conjugated in $G$ to a standard parabolic subgroup.
By conjugation $\sigma $ and $\nu $ in the same manner and
conjugating then $\sigma $ and $M$ inside $M_{\nu  }$, so that $M$
becomes the Levi factor of a standard parabolic subgroup $P$, one
can finally assume $P_{\nu }$ standard and $\nu\in a_{M_{\nu
}}^{*+}$.

One can then write $i_P^G\sigma_{\nu }=i_{P_{\nu }}^G(i_{P\cap
M_{\nu }}^{M_{\nu }}\sigma)_{\nu }$. The representation $\tau
=i_{P\cap M_{\lambda }}^{M_{\lambda }}\sigma $ is a direct sum of
irreducible tempered representations $(\tau _i,E_i)$. (Some of
them may be isomorphic).

Write $\widetilde{P}$ for the standard parabolic subgroup which is
conjugated to $\ol{P}$ by $w$. Put $P_1=\ol{P\cap M_{\nu }}\
U_{\nu }$ and denote by $\ti{P_1}$ the parabolic subgroup of $G$
which is conjugated to $\ol{P_1}$ by $w$.

Denote by $\mc{F}_{\widetilde{P_1}wP_1}$ the subspace of
$i_{\widetilde{P_1}}^GwE$ formed by the functions with support in
the open set $\widetilde{P_1}wP_1$. It follows from the geometric
lemma that the Jacquet functor $j_{P_1}^G$ sends
$\mc{F}_{\widetilde{P_1}wP_1}$ to a subspace of
$j_{P_1}^Gi_{\widetilde{P_1}}^G wE$ on which $M$ acts by the
representation $\sigma _{\nu }$.

Choose a Whittaker function $e$ in the space of $\sigma $ with
nonzero value in $1$ and an open compact subgroup $H$ of $G$ of
Iwahori type with respect to $(P_1,M)$, such that $e$ is $H\cap
M$-invariant. By the lemma {\bf 2.4}, there is an element $v_0$ in
$S_{P_1}^H(i_{\ti{P_1}}^GwE)$ with support in $\ti{P}_1wH$ such
that $v_0(w)=e$. Recall that $\ti{P}_1wH=\ti{P}_1w(H\cap
U)\subseteq \widetilde{P_1}wP_1$. It follows directly from the
definition that $\lambda _{\widetilde{P_1}} (\widetilde{w}\nu
,w\sigma ,\psi )$ does not vanish in $v_0$.

By the lemma \ref{lem:Intertwining}, the intertwining operator
$A_{w}= t(w^{-1})J_{\ol{\ti{P}} \vert\ti{P_1}}(w\sigma _{\nu })$
is well defined and $\lambda _P(\nu ,\sigma,\psi )A_{w}=c\lambda
_{\ti{P_1}}(\ti{w}\nu ,w\sigma ,\psi )$, where $c$ is a non zero
constant. In particular, $\lambda _P(\nu ,\sigma,\psi )$ is non
zero in $A_{w}v_0$. It remains to show that $A_{w}v_0$ lies in the
subspace $(i_{P}^GE_{\nu })_0$ of $i_{P}^GE_{\nu }$ spanned by the
irreducible sub-representations. For this we will show on the one
hand that the Jacquet functor $j_{P_1}^G$ sends $A_{w}v_0$ to a
nonzero element of the subspace $(j_{P_1}^Gi_P^GE_{\nu })_{\nu
}^{\infty }$ of $j_{P_1}^Gi_P^GE_{\nu }$ generated by the
sub-representations which admit a generalized central character
with real part $\nu $. On the other hand we will show that the
Jacquet functor $j_{P_1}^G$ sends the subspace $(i_P^GE_{\nu })_0$
onto $(j_{P_1}^Gi_P^GE_{\nu })_{\nu }^{\infty }$. As $A_{w}v_0$ is
by \cite{H1} proposition {\bf 4.1.1} an element of
$S_{P_1}^H(i_P^GE_{\nu })$, it follows then from lemma {\bf 2.3}
that $A_{w}v_0$ lies in $S_{P_1}^H((i_P^GE_{\nu })_0)$ and
consequently in $(i_P^GE_{\nu })_0$. This finishes the proof.

Let us show first that $j_{P_1}^G$ sends $A_{w}v_0$ to a nonzero
element of $(j_{P_1}^Gi_P^GE_{\nu })_{\nu }^{\infty }$. As
$A_{w}v_0$ is a nonzero element in $S_{P_1}^H(i_P^GE_{\nu })$,
$j_{P_1}^GA_{w}v_0$ is nonzero by the proposition {\bf 2.2}. It is
then enough to show that $T_M$ acts on $j_{P_1}^GA_wv_0 $ by a
character equal to the central character $\chi _{\nu }$ of $\sigma
_{\nu }$. For every $a\in T_M$, $(i_{\ti{P_1}}^G w\sigma _{\nu
})(a)v_0-\chi _{\nu }(a)v_0$ has trivial image in
$j_{P_1}^Gi_{\ti{P_1}}^GwE_{\nu }$, because $j_{P_1}^Gv_0$ lies in
a subspace isomorphic to $\sigma _{\nu }$. This means that there
are $u_1,\dots ,u_t\in U_1$ and $v_1,\dots ,v_t$ in
$i_{\ti{P_1}}wE_{\nu }$, such that
$$(i_{\ti{P_1}}^Gw\sigma _{\nu })(a)v_0-\chi _{\nu
}(a)v_0=\sum _i[(i_{\ti{P_1}}^Gw\sigma _{\nu
})(u_i)v_i-v_i].$$ Applying on both sides $A_{w}$, one gets
$$i_P^G\sigma _{\nu }(a)A_{w}v_0-\chi _{\nu
}(a)A_{w}v_0=\sum _i[(i_P^G\sigma _{\nu })(u_i)A_{w}v_i-
A_{w}v_i].$$ It follows that $(j_{P_1}^Gi_P^G\sigma _{\nu })(a)$
acts on $j_{P_1}^G(A_{w}v_0)$ by the character $\chi _{\nu }(a)$.

It remains to show that the irreducible subspaces $\pi_i$ of
$i_P^G\sigma _{\nu }$ are the only sub-quotients such that
$j_{P_1}^G\pi _i$ admits as exponent a generalized character with
real part $\nu $. As $\nu$ is a regular element of $a_{M_{\nu
}}^{*}$, the length of $(j_{P_1}^Gi_P^GE)_{\nu }^{\infty }$ is by
the geometric lemma equal to the cardinality $l$ of the subset of
$W^M\backslash W^{M_{\nu }}/W^M$ formed by the elements which
stabilize $M$. It equals the length of $(j_{P_1\cap M_{\nu
}}^{M_{\nu }}\tau _{\nu })_{\nu }^{\infty }$. Denote by $l_i$ the
length of $(j_{P_1\cap M_{\nu }}(\tau _i)_{\nu })_{\nu }^{\infty
}$. An irreducible sub-representation $\pi_i$ of $i_P^G\sigma
_{\nu }$ is a sub-representation of some $i_{P_{\nu }}^G(\tau
_i)_{\nu }$. It is enough to show that the length of
$(j_{P_1}^G\pi _i)_{\nu }^{\infty }$ is $\geq l_i$, because the
$l_i$ sum up to $l$.

By the Frobenius reciprocity, one has, $$ \textup{Hom} _G(\pi
_i,i_{P_{\nu }}^G(\tau _i)_{\nu })= \textup{Hom} _M(j_{P_{\nu
}}^G\pi _i,(\tau _i)_{\nu }),$$ which means that $(\tau _i)_{\nu
}$ is a quotient of $j_{P_{\nu }}^G\pi _i$. From the transitivity
of the Jacquet functor, if follows that $j_{P_1\cap M_{\nu
}}^{M_{\nu }}(\tau _i)_{\nu }$ is a sub-quotient of $j_{P_1}^G\pi
_i$. As $(j_{P_1\cap M_{\nu }}^{M_{\nu }}(\tau _i)_{\nu })_{\nu
}^{\infty }$ has length $l_i$, it follows that the length of
$(j_{P_1}^G\pi _i)_{\nu }^{\infty }$ is at least $l_i$.
\end{proof}

\section{Reduction to an affine Hecke algebra setting}

Let $P=MU$ be a maximal standard parabolic subgroup of $G$. Denote
by $\alpha $ the unique simple $F$-root for $G$ which is not a
root for $M$ and by $\rho $ half the sum of the $F$-roots whose
root spaces span Lie$(U)$. Remark that $\rho $ lies in $a_M^*$.
For an $F$-root $\beta $, denote by $\underline{\beta }$ a root in
the absolute root system that restricts to $\beta $ and by
$\underline{\beta }^{\vee }$ the coroot corresponding to
$\underline{\beta }$. Write $\langle\cdot, \cdot\rangle $ for the
duality between $a_T^*$ and $a_T$. For $\lambda\in a_T^*$ and an
$F$-root $\beta $, we will sometimes write $\langle\lambda
,\underline{\beta }^{\vee }\rangle$. Here, $\underline{\beta
}^{\vee }$ will be identified with its orthogonal projection on
$a_T$. Put $\ti{\alpha }=\langle\rho ,\underline{\alpha }^{\vee
}\rangle ^{-1}\rho $. Let $(\tau ,V)$ be an irreducible discrete
series representation of $M$. By proposition 2.5, there is a
standard parabolic subgroup $P_1=M_1U_1$ of $G$ contained in $P$,
a unitary $\psi $-generic irreducible supercuspidal representation
$\sigma $ of $M_1$ and $\nu _{\tau }\in a_{M_1}^{M*}$, $\nu _{\tau
}\geq _{M\cap P_1}0$, such that $\tau $ is a sub-representation of
$i_P^G(\sigma\otimes\chi _{\nu _{\tau } })$. (Remark that we do
not need for the sequel such a strong result, but only the well
known existence of a generic supercuspidal support.)

Denote by $\Sigma _{red} (P_1)$ the set of reduced roots for the
action of the split center of $M_1$ on $Lie(U_1)$. Remark that to
any $\ol{\beta }\in\Sigma _{red} (P_1)$ one can associate a
parabolic subgroup $P_{1,\ol{\beta }}=M_{1,\ol{\beta
}}U_{1,\ol{\beta }}$, such that $P_1\cap M_{1,\ol{\beta }}$ is a
maximal standard parabolic subgroup of $M_{1,\ol{\beta }}$. For
$\ol{\beta }\in\Sigma _{red}(P_1)$, we will denote $\beta $ the
unique simple root for $M_{1,\ol{\beta }}$ which projects to
$\ol{\beta }$ and write then also $M_{1,\beta }$, $U_{1,\beta }$
and $P_{1,\beta }$.

Harish-Chandra's $\mu$-function $\mu (\sigma\otimes\chi _{\nu })$
is a product
$$\prod _{\ol{\beta }\in\Sigma _{red} (P_1)}\mu ^{M_{1,\beta } }
(\sigma\otimes\chi _{\nu }).$$ The set of roots $\ol{\beta }$ such
that $\mu ^{M_{1,\beta }} (\sigma\otimes\chi _{\nu })$ is not
holomorphic on $a_{M_1}^*$ as a function in $\nu $ is the set of
positive roots of a root system in $a_{M_1}^*$ (cf. \cite{Si2},
proposition 3.5). We will denote this root system by $\Sigma
_{\sigma }$. Denote by $\overline{\beta } ^{\vee }$ the coroot of
a root $\ol{\beta }$ in $\Sigma _{\sigma }$. Remark that by the
main result of \cite{H2}, $\nu _{\tau }$ is a residue point in
$a_{M_1}^{M*}$ for Harish-Chandra's $\mu $-function
$\nu\mapsto\mu^M (\sigma\otimes\chi _{\nu })$, defined relative to
$M$. (The precise definition of a residue point, which is given in
\cite{O}, does not matter here.)

Fix $\ol{\beta }\in\Sigma _{\sigma }$. In \cite{Sh3}, F. Shahidi
has associated to each irreducible component $r_{1,i}$ of the
adjoint action of $^LM_1$ on $Lie(\ ^LU_{1,\beta })$ a meromorphic
function $\gamma ^{M_{1,\beta }}(s,\sigma,r_{1,i},\psi )$. He
showed that there is at most one index $i$ such that $\gamma
^{M_{1,\beta }}(s,\sigma ,r_{1,i},\psi )$ has a zero on the real
axis and that this index equals in fact either $1$ or $2$. We will
denote it in the sequel by $\epsilon _{\beta }$, put $\epsilon
_{\ol{\beta }}= {\langle\ol{\beta },\underline{\beta }^{\vee
}\rangle\over 2}\epsilon _{\beta }$ and $i_{\overline{\beta
}}=\langle\ti{\alpha },\bar{\beta }^{\vee }\rangle$.

\null
\begin{prop} There are meromorphic functions $f$ and $f_i$ which are holomorphic
and non-vanishing on the real axis, such that
$$C_{\psi }(s\ti{\alpha },\tau )=f(s)\prod _{\ol{\beta }
\in\Sigma_{\sigma }^+-{\Sigma ^M_{\sigma }}^+}{1-q^{-\langle\nu _{\tau }
+s\widetilde{\alpha },\ol{\beta }^{\vee }\rangle}\over
1-q^{-{1\over\epsilon _{\ol{\beta }}}+\langle\nu _{\tau }
+s\widetilde{\alpha },\ol{\beta }^{\vee }\rangle}}$$
$$\gamma (is,\tau , r_i, \psi )=f_i(s)\prod _{\ol{\beta
},\ \epsilon_{\ol{\beta }}i_{\ol{\beta }}=i}{1-q^{-\langle\nu _{\tau }
+s\widetilde{\alpha },\ol{\beta }^{\vee }\rangle}\over
1-q^{-{1\over\epsilon _{\ol{\beta }}}+\langle\nu _{\tau }
+s\widetilde{\alpha },\ol{\beta }^{\vee }\rangle}}.\leqno{\it and}$$
\end{prop}

\begin{proof} Denote by $r$ the adjoint action of $^LM$ on $V=Lie( ^LU)$.
This action decomposes in irreducible sub-representations $r_i$
corresponding to the weights of $T_{^LM}$. The space $V_i$ of
$r_i$ is generated by the root spaces $n_{\underline{\beta }^
{\vee }}$ corresponding to the roots $\underline{\beta }^{\vee }$
which have the same restriction to $T_{^LM}$ as
$i\underline{\alpha }^{\vee }$.

The local coefficient $C_{\psi }$ can be expressed by the $\gamma
$-function defined in \cite{Sh3}: up to a product by a holomorphic
function, $C_{\psi }(s\widetilde{\alpha },\tau )$ equals $\prod
_i\gamma (is,\tau, r_i,\psi )$ (cf. identity (3.11) in
\cite{Sh3}). Write $\gamma (is,\tau ,r_i,\psi )=\gamma
(\tau\otimes\chi _{s\widetilde{\alpha }},r_i,\psi )$. For
$\beta\in\Sigma _{red}(P_1)$ denote by $r_{1,i,\beta }$ the
restriction of $r_i$ to $^LM_1\rightarrow Lie(\ ^LU_{1,\beta })$.
Then, by the product formula for the $\gamma $-function (cf.
identity (3.13) in \cite{Sh3}), one has
$$\gamma (\tau\otimes \chi _{s\widetilde{\alpha }},r_i,\psi
)=\prod _{\ol{\beta }} \gamma ^{M_{1,\beta }}(\sigma\otimes \chi_{\nu
_{\tau } +s\widetilde{\alpha }}, r_{1,i,\beta },\psi ),$$ the
roots $\ol{\beta }$ being taken in $\Sigma _{red}(P_1)-\Sigma
_{red}(P_1\cap M)$.

Define $i_{\beta }=\langle\ti{\alpha },\underline{\beta }^{\vee
}\rangle $. The representation $r_{1,i,\beta }$ can only be
nonzero if $i_{\beta }\vert i$. Then, $\gamma ^{M_{1,\beta
}}(\sigma\otimes \chi _{\nu _{\tau } +s\widetilde{\alpha
}},r_{1,i,\beta },\psi )$ is equal to $\gamma ^{M_{1,\beta
}}(\sigma\otimes \chi _{\nu _{\tau } +s\widetilde{\alpha
}},r_{i\over i_{\beta }},\psi )$. This function is holomorphic and
nonzero for $s\in\mathbb R$, except perhaps if $i=\epsilon _{\beta
}i_{\beta }$ with $\epsilon _{\beta }\in\{1,2\}$. This can then
only happen at one of these two values for $\epsilon _{\beta }$
(cf. Corollary {\bf 7.6} of \cite{Sh3}). Then $\gamma ^{M_{1,\beta
}}(\sigma\otimes \chi _{\nu _{\tau } +s\widetilde{\alpha
}},r_{i\over i_{\beta }},\psi )$ is equal to the product of a
function in $s$ which is holomorphic and non-vanishing on the real
axis by $L(1-\epsilon_{\beta }\langle\nu _{\tau }+s\ti{\alpha }
,\underline{\beta }^{\vee }\rangle,\sigma^{\vee
},r_i)/L(\epsilon_{\beta }\langle\nu _{\tau }+s\ti{\alpha }
,\underline{\beta }^{\vee }\rangle,\sigma,r_i)$ (cf. identity
(7.4) of \cite{Sh3}). Up to a product by a holomorphic
non-vanishing function on the real axis, this quotient equals
$(1-q^{-\epsilon _{\beta }\langle\nu _{\tau }+s\ti{\alpha
},\underline{\beta }^{\vee }\rangle})/(1-q^{-1+\epsilon _{\beta
}\langle\nu _{\tau }+s\ti{\alpha },\underline{\beta }^{\vee
}\rangle})$.

Denote by $\Sigma '$ the subset of the roots $\ol{\beta}\in\Sigma
_{red}(P_1)\setminus\Sigma _{red}(P_1\cap M)$ such that $\gamma
^{M_{1,\beta }}(\sigma\otimes \chi _{\nu _{\tau }
+s\widetilde{\alpha }},r_{i\over i_{\beta }},\psi )$ has a pole or
a zero in some $s\in\mathbb R$. We have just proved that $C_{\psi
}(s\ti{\alpha },\tau )$ is, up to the product by a meromorphic
function without poles and zeroes on the real axis, equal to
$$\prod _{\ol{\beta}\in\Sigma '}{1-q^{-\epsilon _{\beta } \langle
\nu _{\tau }+s\ti{\alpha },\underline{\beta }^{\vee }\rangle}\over
1-q^{-1+\epsilon _{\beta }\langle\nu_{\tau }+s\ti{\alpha }
,\underline{\beta} ^{\vee }\rangle}}.$$ This
expression is the product of a regular function on $\mathbb R$
depending on $s$ without zeroes on the real axis by
$$\prod _{\ol{\beta }\in\Sigma'}{1-q^{-\langle\nu _{\tau }
+s\widetilde{\alpha },\underline{\beta}^{\vee }\rangle}\over
1-q^{-{1\over\epsilon _{\beta }}+\langle\nu _{\tau }
+s\widetilde{\alpha },\underline{\beta}^{\vee }\rangle}}.$$

Recall that Harish-Chandra's $\mu $-function $\mu
(\sigma\otimes\chi _{\nu })$ is a product $$\prod _{\ol{\beta }
\in\Sigma _{red} (P_1)}\mu ^{M_{1,\beta }} (\sigma\otimes\chi
_{\nu }).$$ The factor $\mu ^{M_{1,\beta } } (\sigma\otimes\chi
_{\nu })$ has a zero or a pole in $\nu\in a_{M_1}^*$, if and only
if $\ol{\beta }\in \Sigma_{\sigma }^+$. Then, there is a positive
real number $\epsilon _{\ol{\beta }}$, such that $\mu ^{M_{1,\beta
} } (\sigma\otimes\chi _{\nu })$ is the product of a function
without zeros and poles on $a_{M_1}^*$ by $\prod _{\beta
'\in\{\pm\ol{\beta }\}}{1-q^{\langle\nu ,{\beta '} ^{\vee
}\rangle}\over 1-q^{-{1\over\epsilon _{\overline{\beta }
}}+\langle\nu ,{\beta '} ^{\vee }\rangle}}$ (cf. \cite{Si1},
theorem 1.6). From the formula relating the $\mu -$function and
the local coefficient $C_{\psi }$ (cf. \cite{Sh3}, identity 1.4),
applied to $\mu ^{M_{1,\beta }}$ for each $\beta $, and the above
relation between local coefficient $C_{\psi }$ and the $\gamma
$-function, applied to $\gamma ^{M_{1,\beta }}$ for each $\beta $,
it follows that $\Sigma '=\Sigma _{\sigma }^+-\Sigma
_{red}(P_1\cap M)$. One deduces from this also that, for
$\beta\in\Sigma '$, the functions $s\mapsto -{1\over\epsilon
_{\ol{\beta }}}+\langle \nu_{\tau }+s\ti{\alpha },\ol{\beta
}^{\vee }\rangle$ and $s\mapsto -{1\over\epsilon _{\beta
}}+\langle \nu_{\tau }+s\ti{\alpha },\underline{\beta }^{\vee
}\rangle$ must have the same zeroes on the real axis. As
$\ol{\beta }^{\vee }$ is a scalar multiple of the projection of
$\beta ^{\vee }$ to $a_{M_1}$, it follows that $\epsilon
_{\overline{\beta }}$ is equal to the product of $\epsilon _{\beta
}$ by $\langle\ol{\beta },\underline{\beta }^{\vee }\rangle\over
2$.

Going back to the expressions for the $\gamma $-factors and
remarking that $\epsilon_{\beta }i_{\beta }=\epsilon _{\ol{\beta
}}i_{\ol{\beta }}$, one gets the statement for the different
$\gamma $-factors.
\end{proof}

\section{The conjectures for affine Hecke algebras}

Let $\Sigma $ be a reduced root system in a vector space
$a_{M_1}^*$. Let $a_{M_1}^{M*}$ be a subspace of codimension one,
generated by a subset $\Sigma ^{M+}$ of positive roots in a
standard sub-root system $\Sigma ^M$ of $\Sigma $. For each
positive root $\beta\in\Sigma $, let $\epsilon _{\beta }$ be a
number $>0$ such that $\epsilon _{\beta }=\epsilon _{\alpha }$ if
$\beta $ and $\alpha $ are conjugated.

Let $\mu $ be the meromorphic function on $a_{M_1}^*$ in $\nu $
defined by
$$\prod _{\beta\in\Sigma }{1-q^{\langle\nu ,\beta ^{\vee
}\rangle}\over 1-q^{-{1\over \epsilon _{\beta }}+\langle\nu ,\beta
^{\vee }\rangle}},$$ and let $\mu ^{M}$ be the factor of $\mu $
given by $$\prod _{\beta\in\Sigma ^{M}}{1-q^{\langle\nu ,\beta
^{\vee }\rangle}\over 1-q^{-{1\over \epsilon _{\beta }}+\langle\nu
,\beta ^{\vee }\rangle}}.$$

Let $\nu _{\tau }$ be a residue point \cite{O} for $\mu ^{M}$ in
$a_{M_1}^{M*}$. Denote by $\omega_{\alpha }$ the fundamental
weight in $a_{M_1}^*$, which corresponds to the simple root
$\alpha $ of $\Sigma $ which does not lie in $a_{M_1}^{M*}$.
Consider the functions $$C(s)=\prod _{\beta\in\Sigma^+-\Sigma
^{M+}}{1-q^{-\langle\nu _{\tau } +s\omega_{\alpha },\beta ^{\vee
}\rangle}\over 1-q^{-{1\over\epsilon _{\beta }}+\langle\nu _{\tau
} +s\omega _{\alpha },\beta ^{\vee }\rangle}}$$
$$\gamma _i(s\omega _{\alpha },\tau ,\psi )=f_i(s)\prod _{\beta ,
\epsilon_{\beta }\langle\omega _{\alpha },\beta ^{\vee }\rangle =i}{1-q^{-\langle\nu _{\tau }
+s\omega _{\alpha },\beta ^{\vee }\rangle}\over
1-q^{-{1\over\epsilon _{\beta }}+\langle\nu _{\tau }
+s\omega_{\alpha },\beta ^{\vee }\rangle}}.\leqno{\rm and}$$

\begin{thm} \it For each irreducible component
of $\Sigma $, suppose either that all the labels $\epsilon _{\beta
}$ are equal, or that ${\epsilon _{\beta '}/\epsilon _{\beta }}$
equals the ratio of the square of the lengths of $\beta '$ and
$\beta $.

Then the function $C(s)$ is holomorphic for $s<0$ and the
functions $\gamma _i(s)$ are non-vanishing for $s>0$.
\end{thm}
\begin{proof} Suppose first all $\epsilon _{\beta }=1$. Denote by
$G_{\Sigma }$ the group of $F$-points of a split connected
reductive group defined over $F$ with root system $\Sigma $ and by
$B_{\Sigma }=T_{\Sigma }U_{\Sigma }$ a Borel subgroup which is
standard with respect to the choice of the ordering of $\Sigma $.
Then $\Sigma ^M$ corresponds to a standard maximal parabolic
subgroup $P=MU$ of $G_{\Sigma }$. As $\nu _{\tau }$ is a residue
point, the representation $i_{B\cap M}^M\chi _{\nu_{\tau }}$ has a
sub-quotient which is a discrete series representation \cite{H2}.
By \cite{MSh}, proposition 3.1, it has also a generic discrete
series sub-quotient. There is an element $w$ in the Weyl group for
$M$, such that $\tau $ is a sub-representation of $i_{B\cap
M}^M\chi _{w\nu _{\tau }}$. By \cite{MSh}, $C_{\psi
}(s\omega_{\alpha })$ is holomorphic for real $s<0$. By
proposition 3.1, $C_{\psi }(s\omega_{\alpha })$ is, up to a factor
which is holomorphic and non-vanishing on the real line, equal to
$$C_{\psi }(s\omega_{\alpha },\tau )=f(s)\prod
_{\beta\in\Sigma^+-\Sigma ^{M+}}{1-q^{-\langle w\nu _{\tau }
+s\omega_{\alpha },\beta ^{\vee }\rangle}\over 1-q^{-1+\langle
w\nu _{\tau } +s\omega_{\alpha },\beta ^{\vee }\rangle}}.$$ As $w$
leaves the set $\Sigma ^+-\Sigma ^{M+}$, the element
$\omega_{\alpha }$ and the product $\langle .,. \rangle$
invariant, the statement follows. The set over which factors the
function $\gamma _i$ is also invariant by the Weyl group of $M$.
As the numerator of $\gamma _i$ is just the reciprocal of the
i$th$ $L$-function of $\tau $, its non-vanishing property follows
from the holomorphicity of the corresponding $L$-function proved
in \cite{MSh}.

Denote by $z_n(s)$ (resp. $z_p(s)$) the number of roots $\beta
\in\Sigma^+-\Sigma ^{M+}$, such that $\langle\nu _{\tau }
+s\omega_{\alpha },\beta ^{\vee }\rangle=0$ (resp. $\langle\nu
_{\tau } +s\omega_{\alpha },\beta ^{\vee }\rangle={1\over \epsilon
_{\beta }}$) and $z_{n,i}$ (resp. $z_{p,i}$) the subsets
corresponding to the roots $\beta $ such that $\langle\omega
_{\alpha },\beta ^{\vee }\rangle =i$. The holomorphicity of
$C(-s)$ in $s$ is equivalent to $z_n(-s)\geq z_p(-s)$ and the
non-vanishing of $\gamma _i(s)$ to $z_{n,i}(s)\leq z_{p,i}(s)$. By
what we just remarked this is true for $s>0$, when all the
$\epsilon _{\beta }$ are equal to $1$.

Suppose now all $\epsilon _{\beta }$ equal to an $\epsilon >0$.
Multiplying the equations above by $\epsilon $, $z_n(s)$ is the
number of roots $\beta \in\Sigma^+-\Sigma ^{M+}$, such that
$\langle\epsilon\nu _{\tau } +\epsilon s\omega_{\alpha },\beta
^{\vee }\rangle=0$, and $z_p(s)$ the number of roots $\beta
\in\Sigma^+-\Sigma ^{M+}$, such that $\langle\epsilon\nu _{\tau }
+\epsilon s\omega_{\alpha },\beta ^{\vee }\rangle=1$. Observe
that, if $\nu _{\tau }$ is a residue point for all $\epsilon
_{\beta }=\epsilon $, then $\epsilon\nu _{\tau }$ is a residue
point for all $\epsilon _{\beta }=1$. Consequently, we are in the
situation of equal parameters $1$, where the holomorphicity and
non-vanishing results have just been proved.

Suppose now $\Sigma $ of type $B_n$, $C_n$, $F_4$ or $G_2$. Denote
by $\kappa $ the ratio of the square of the length of a long root
by the one of a short root. Suppose $\epsilon _{\beta '}/\epsilon
_{\beta }=\kappa $, if $\beta '$ is a long root and $\beta $ a
short root. Write $\ti{\beta }=\beta/\kappa $, if $\beta $ is a
long root, $\ti{\beta }=\beta $, if $\beta $ is a short root, and
denote by $\ti{\Sigma }$ the set of the $\ti{\beta }$. Then
$\ti{\Sigma }$ is a root system of type $C_n$, if $\Sigma $ was of
type $B_n$, of type $B_n$, if $\Sigma $ was of type $C_n$, and of
type $F_4$ (resp. $G_2$), if $\Sigma $ was of type $F_4$ (resp.
$G_2$). Let $\epsilon $ be the common value of the $\epsilon
_{\beta }$ with $\beta $ a short root. Then, $z_n(s)$ is the
number of roots $\ti{\beta }\in \ti{\Sigma}^+ -\ti{\Sigma }^{M+}$,
such that $\langle\nu _{\tau } +s\omega_{\alpha },\ti{\beta
}^{\vee }\rangle=0$, and $z_p(s)$ the number of roots $\beta
\in\ti{\Sigma }^+-\ti{\Sigma }^{M+}$, such that $\langle\nu _{\tau
} +s \omega_{\alpha },\ti{\beta } ^{\vee }\rangle=1/\epsilon $.
Remark that $\nu _{\tau }$ is a residue point for the set of roots
$\ti{\Sigma }^M$ with all labels equal $\epsilon$. So, we are back
in the equal parameter case, where the holomorphicity  and
non-vanishing result have already been considered above, adding
that $\epsilon _{\ti{\beta }}\langle\omega_{\alpha },\ti{\beta
}^{\vee }\rangle=\epsilon _{\beta }\langle\omega_{\alpha },\beta
^{\vee }\rangle $.
\end{proof}

\section{The Conjectures in the p-adic case}

Recall that $P=MU$ denotes a maximal standard parabolic subgroup
of $G$, $\alpha $ the unique simple $F$-root for $G$ which is not
a root for $M$, $\rho $ half the sum of the $F$-roots whose root
spaces span Lie$(U)$ and that $\ti{\alpha }=\langle\rho
,\underline{\alpha }^{\vee }\rangle ^{-1}\rho $.

\begin{thm} Let $(\tau ,V)$ be an irreducible
tempered representation of $M$. The function $C_{\psi
}(-s\widetilde{\alpha },\tau )$ and the functions $L(\tau, s,
r_i)$ are regular for $s>0$.
\end{thm}

\begin{proof} By the product formula for the local coefficient
$C_{\psi }$ and the $\gamma-$functions, one is reduced to consider
the case, where $\tau $ is a discrete series representation. Here
the theorems 3.1 and 4.1 apply. So, it remains to show that the
labels $\epsilon _{\ol{\beta }}$ satisfy the assumption in the
statement of the theorem 4.1. Denote by $\Sigma $ the reduced
$F$-root system for $G$, by $P_1=M_1U_1$ and $\sigma $
respectively the standard parabolic subgroup and the generic
supercuspidal representation of $M_1$ from which $\tau $ is
induced and by $\Sigma ^{M_1}$ the reduced $F$-root system of the
Levi subgroup $M_1$.

One has to show that for two roots $\beta '$ and $\beta $ in
$\Sigma_{\sigma }$ the quotient $\epsilon _{\ol{\beta '}}/\epsilon
_{\ol{\beta }}$ satisfies the assumptions in the statement of
theorem 4.1. We will prove first that one can reduce to the case
where $\Sigma $ is irreducible and $\Sigma ^{M_1}$ of corank 2.

Remark that the labels $\epsilon _{\ol{\beta '}}$ and $\epsilon
_{\ol{\beta }}$ do not change if one conjugates $\beta '$ and
$\beta $ by an element of the Weyl group of $\Sigma _{\sigma }$.
So, we may suppose that $\beta +\beta '$ is a root in $\Sigma
_{\sigma }$. Suppose that the corank of $\Sigma ^{M_1}$ in $\Sigma
$ is $>2$ and denote by $\Sigma ^{M'}$ the sub-root system of
$\Sigma $ of the minimal Levi sub-group $M'$ of $G$ containing
$\Sigma ^{M_1}$, $\beta $ and $\beta '$. Then, possibly after
conjugation, $\Sigma ^{M_1}$ is a standard corank 2 sub-root
system in $\Sigma ^{M'}$ and the values of the numbers defined in
the proposition 3.1 are the same with respect to $\Sigma ^{M'}$ or
$\Sigma $. If $\Sigma ^{M'}$ is not irreducible, then $\beta $ and
$\beta '$ must be projections of roots in a same irreducible
component $\Sigma _1$ of $\Sigma $, because $\beta +\beta '$ is a
root in $\Sigma _{\sigma }$. The system $\Sigma ^{M_1}\cap\Sigma
_1$ is a sub-root system of corank 2 in $\Sigma _1$, and one is
reduced to study the subgroup of $G$ generated by $\Sigma ^{M_1}
\cap \Sigma _1$ relative to the one generated by $\Sigma _1$ with
the restriction of $\sigma $ to this subgroup. So, one is finally
reduced to the case, where $\Sigma ^{M_1}$ is a sub-root system of
corank 2 of $\Sigma $. This situation is considered case by case
in the next section, using the following lemma.\end{proof}

\begin{lem} Denote by $(.,.)$ the Weyl group-invariant
scalar product in the space spanned by the absolute roots of $G$
and, for a root $\overline{\beta }$ in $\Sigma _{\sigma }$, by
$\omega ^{M_{1,\beta }}_{\beta }$  the fundamental weight
corresponding to $\beta $ relative to the root system $\Sigma
^{M_{1,\beta }}$ and by $\ti{\beta }$ the scalar multiple of
$\omega ^{M_{1,\beta }}_{\beta }$ that verifies $\langle\ti{\beta
},\underline{\beta }^{\vee }\rangle =1$.

The labels $\epsilon _{\ol{\beta '}}, \epsilon _{\ol{\beta }}$ and
$\epsilon _{\beta '}, \epsilon _{\beta }$ defined in section 3
verify the formula
$${\epsilon _{\ol{\beta '}}\over\epsilon _{\ol{\beta }}}={\epsilon _{\beta '}
(\underline{\beta '},\underline{\beta '})(\ti{\beta },\ti{\beta })\over
\epsilon _{\beta }(\underline{\beta },\underline{\beta })(\ti{\beta '},
\ti{\beta '})}.$$
\end{lem}

\begin{proof} Recall that $\epsilon _{\ol{\beta
}}={\langle\ol{\beta },\underline{\beta }^{\vee }\rangle\over 2}
\epsilon _{\beta }$. So, it is enough to show that $$\ol{\beta
}={(\underline{\beta },\underline{\beta }) \over 2(\ti{\beta
},\ti{\beta })}\ti{\beta }.$$ Remark first that for every $\lambda
$ in $a_T^*$ and every root $\gamma $, one has $\langle\lambda
,\underline{\gamma }^{\vee }\rangle={2\over (\underline{\gamma
},\underline{\gamma })}(\lambda ,\underline{\gamma })$. It is
clear that $\ol{\beta }=\kappa\ti{\beta }$ for some constant
$\kappa$, because both lie in the one-dimensional vector space
$a_{M_1}^{M_{1,\beta }*}$. Then, one computes
$$\kappa ={(\ti{\beta },\ol{\beta })\over (\ti{\beta
},\ti{\beta })}={(\ti{\beta },\underline{\beta })\over (\ti{\beta
},\ti{\beta })}={(\underline{\beta },\underline{\beta })\over
2(\ti{\beta },\ti{\beta })}\langle\ti{\beta },\underline{\beta
}^{\vee }\rangle={(\underline{\beta },\underline{\beta })\over
2(\ti{\beta },\ti{\beta })}.$$
\end{proof}

\section{Labels of supercuspidal $\mu $-functions in the generic
case}

Remark first that our situation is invariant for restriction of
scalars: if $\underline{H}$ is a quasi-split connected reductive
$F$-group, $F'/F$ a finite Galois extension and $\underline{G}=
Res_{F'/F}\underline{H}$, then the absolute root system for
$\underline{G}$ is a union of copies of the absolute root system
of $\underline{H}$ with an action of the Galois group permuting
these copies. In particular, the absolute roots (resp. duals of
the absolute roots) for $\underline{G}$ restrict to $F$-roots as
do the absolute roots (resp. duals of the absolute roots) for
$\underline{H}$. So, as every $F$-quasi-simple group G is the
restriction of scalars of an absolute quasi-simple group, it is
enough to consider the latter ones. (Of course, in the split case,
this does not make any difference.)

In this section, we give for every absolute root system of an
absolute quasi-simple quasi-split group over $F$ its
Dynkin-diagram, its $F$-root system $\Sigma $, the list of the
standard sub-root systems $\Sigma ^M$ of corank 2 of $\Sigma $ and
the set of quotient roots $\Sigma (T_M)$. We consider then the
subset $\Sigma _{\mu }$ formed by the roots $\beta $ in $\Sigma
(T_M)$ such that $\Sigma ^M$ is self-conjugated as a corank one
sub-root system of $\Sigma ^{M_{\beta }}$. It turns out that
$\Sigma _{\mu }$ is always a root system, and it is clear that any
root system $\Sigma _{\sigma }$ which may appear from the above
context must be a sub-root system of $\Sigma _{\mu }$.

One does not have to study further the cases where $\Sigma _{\mu
}$ is a product of irreducible root systems of type $A$, because
in this case all roots which lie in a same irreducible component
are conjugated. So, only the cases where $\Sigma _{\mu }$ is of
type $B_2$ or $G_2$ will require further attention. We call these
cases the \it relevant cases. \rm With help of lemma 5.2, we
compute in these cases the possible values of the labels $\epsilon
_{\overline{\beta }}$ corresponding to the long and short root,
using the list in \cite{L} completed in \cite{Sh2}. In some cases,
we will need in addition the following lemma to prove that
unwanted ratios for the labels do not appear.

\begin{lem} Let $\sigma $ be a generic supercuspidal
representation of a maximal Levi subgroup $M'$ of a quasi-split
connected reductive group $G'$ defined over $F$. The second
$L$-function $L(s,\sigma ,r_2)$ attached to $\sigma $ is constant
in the following cases:

(i) $G'$ is split of type $D_5$ and $M'$ is of type $A_2\times
A_1\times A_1$,

(ii) $G'$ is split of type $D_7$ and $M'$ is of type $A_2\times
D_4$,

(iii) $G'$ is split of type $C_3$ and $M'$ is of type $A_2$.

(iv) $G'$ is quasi-split of type $^2A_5$ and $M'$ is the
restriction of scalars of a group of type $A_2$ relative to a
cyclic extension of $F$ of degree $2$.

(v) $G'$ is quasi-split of type $^2D_4$ and $M'$ a split group of
type $A_2$.
\end{lem}

\begin{proof} The second $L$-function is here in fact the one
attached to the exterior square L-function of the $A_2$ part which
can be reinterpreted as the first and only $L$-function in a non
associated setting. So it follows from \cite[lemma 7.4]{Sh3} that
the $L$-function is $1$.
\end{proof}

We will denote in the sequel abusively still by $\alpha _i$ the
restriction of a relative root $\alpha _i$, if it is non trivial.

\subsection{The split cases:} Here $\ti{\beta }$ is always equal
to the fundamental weight $\omega _{\beta }^{M_{\beta }}$ in
$\Sigma ^{M_{\beta }}$.

\null $\mathbf A_n:$

\input{An.TpX}

$\Delta -\Delta^M=\{\alpha _i,\alpha _j\}$, $1\leq i<j\leq n$, $M$
is of type $A_{i-1}\times A_{j-i-1}\times A_{n-j}$, $\Sigma
_{red}(T_M)=\{\alpha _i,\alpha _j,\alpha _i+\alpha _j\}$, so that
$\Sigma _{\mu }$ is always a product of root systems of type $A$.
Consequently, there are no relevant cases.

\null $\mathbf B_n:$

\input{Bn.TpX}

\begin{enumerate}

\item

$\Delta -\Delta^M=\{\alpha _i,\alpha _j\}$, $1\leq i<j\leq
n-1$, $M$ is of type $A_{i-1}\times A_{j-i-1}\times B_{n-j}$,
$\Sigma _{red}(T_M)=\{\alpha _i,\alpha _j,\alpha _i+\alpha
_j,\alpha _i+2\alpha _j\}$ is of type $B_2$, $\alpha _i$ is
the long root, $M_{\alpha _i}$ is of type $A_{j-1}\times
B_{n-j}$ and $M_{\alpha _j}$ is of type $A_{i-1}\times
B_{n-i}$.  In order of $\Sigma _{\mu }$ to be of type $B_2$,
$M$ must be self-conjugate in $M_{\alpha _i}$, which means
that $j=2i$. Then $(\omega ^{M_{\alpha _i}}_{\alpha _i},\omega
^{M_{\alpha _i}}_{\alpha _i})=j/2$, $(\omega ^{M_{\alpha
_j}}_{\alpha _j},\omega ^{M_{\alpha _j}}_{\alpha _j})=j$,
$(\alpha _i,\alpha _i)=(\alpha _j,\alpha _j)=2$, $\epsilon
_{\alpha _i}$ is necessarily $1$ and $\epsilon _{\alpha _j}$
may be $1$ or $2$. One deduces that the assumptions are
satisfied.

\item

$\Delta -\Delta^M=\{\alpha _i,\alpha _n\}$, $1\leq i<n$, $M$
is of type $A_{i-1}\times A_{n-i-1}$, $\Sigma
_{red}(T_M)=\{\alpha _i,\alpha _n,\alpha _i+\alpha _n,\alpha
_i+2\alpha _n\}$ is of type $B_2$, $\alpha _i$ is the long
root. In order of $\Sigma _{\mu }$ to be of type $B_2$, $M$
must be self-conjugate in $M_{\alpha _i}$, which means that
$n=2i$. Then $(\omega ^{M_{\alpha _i}}_{\alpha _i},\omega
^{M_{\alpha _i}}_{\alpha _i})=n/2$, $(\omega ^{M_{\alpha
_n}}_{\alpha _n},\omega ^{M_{\alpha _n}}_{\alpha _n})=n/4$,
$(\alpha _i,\alpha _i)=2$, $(\alpha _n,\alpha _n)=1$,
$\epsilon _{\alpha _i}$ and $\epsilon _{\alpha _n}$ are
necessarily $1$. One deduces that the assumptions are
satisfied.

\end{enumerate}

\null $\mathbf C_n:$

\input{Cn.TpX}

\begin{enumerate}

\item

$\Delta -\Delta^M=\{\alpha _i,\alpha _j\}$, $1\leq i<j\leq
n-1$, $M$ is of type $A_{i-1}\times A_{j-i-1}\times C_{n-j}$,
$\Sigma _{red}(T_M)=\{\alpha _i,\alpha _j,\alpha _i+\alpha
_j,\alpha _i+2\alpha _j\}$ is of type $B_2$, $\alpha _i$ is
the long root, $M_{\alpha _i}$ is of type $A_{j-1}\times
C_{n-j}$ and $M_{\alpha _j}$ is of type $A_{i-1}\times
C_{n-i}$. In order of $\Sigma _{\mu }$ to be of type $B_2$,
$M$ must be self-conjugate in $M_{\alpha _i}$, which means
that $j=2i$. Then $(\omega ^{M_{\alpha _i}}_{\alpha _i},\omega
^{M_{\alpha _i}}_{\alpha _i})=j/2$, $(\omega ^{M_{\alpha
_j}}_{\alpha _j},\omega ^{M_{\alpha _j}}_{\alpha _j})=j$,
$(\alpha _i,\alpha _i)=(\alpha _j,\alpha _j)=2$, $\epsilon
_{\alpha _i}$ is necessarily $1$ and $\epsilon _{\alpha _j}$
may be $1$ or $2$. One deduces that the assumptions are
satisfied.

\item

$\Delta -\Delta^M=\{\alpha _i,\alpha _n\}$, $1\leq i<n$, $M$
is of type $A_{i-1}\times A_{n-i-1}$, $\Sigma
_{red}(T_M)=\{\alpha _i,\alpha _n,\alpha _i+\alpha _n,2\alpha
_i+\alpha _n\}$ is of type $B_2$, $\alpha _n$ is the long
root. In order of $\Sigma _{\mu }$ to be of type $B_2$, $M$
must be self-conjugate in $M_{\alpha _i}$, which means that
$n=2i$. Then  $(\omega ^{M_{\alpha _i}}_{\alpha _i},\omega
^{M_{\alpha _i}}_{\alpha _i})=n/2$, $(\omega ^{M_{\alpha
_n}}_{\alpha _n},\omega ^{M_{\alpha _n}}_{\alpha _n})=n$,
$(\alpha _i,\alpha _i)=2$, $(\alpha _n,\alpha _n)=4$,
$\epsilon _{\alpha _i}$ is necessarily $1$ and $\epsilon
_{\alpha _n}$ may be $1$ or $2$. One deduces that the
assumptions are satisfied.

\end{enumerate}

$\mathbf {D_n}:$

\input{Dn.TpX}

\begin{enumerate}

\item

$\Delta -\Delta^M=\{\alpha _i,\alpha _j\}$, $1\leq i<j\leq
n-2$, $M$ is of type $A_{i-1}\times A_{j-i-1}\times D_{n-j}$,
$\Sigma _{red}(T_M)=\{\alpha _i,\alpha _j,\alpha _i+\alpha
_j,\alpha _i+2\alpha _j\}$ is of type $B_2$, $\alpha _i$ is
the long root, $M_{\alpha _i}$ is of type $A_{j-1}\times
D_{n-j}$, $M_{\alpha _j}$ is of type $A_{i-1}\times D_{n-i}$.
In order of $\Sigma _{\mu }$ to be of type $B_2$, $M$ must be
self-conjugate in $M_{\alpha _i}$, which means that $j=2i$.
Then $(\omega ^{M_{\alpha _i}}_{\alpha _i},\omega ^{M_{\alpha
_i}}_{\alpha _i})=j/2$, $(\omega ^{M_{\alpha _j}}_{\alpha
_j},\omega ^{M_{\alpha _j}}_{\alpha _j})=j$, $(\alpha
_i,\alpha _i)=(\alpha _j,\alpha _j)=2$, $\epsilon _{\alpha
_i}$ is necessarily $1$ and $\epsilon _{\alpha _j}$ may be $1$
or $2$. One deduces that the assumptions are satisfied.

\item

$\Delta -\Delta^M=\{\alpha _i,\alpha _j\}$, $1\leq i<n$,
$j=n-1$ or $j=n$, $M$ is of type $A_{i-1}\times A_{n-i-1}$,
$\Sigma _{red}(T_M)=\{\alpha _i,\alpha _j,\alpha _i+\alpha
_j\}$ is of type $A_2$. Consequently, there are no relevant
cases.

\end{enumerate}

$\mathbf {E_6}:$

Here the only relevant case is

\input{E6.TpX}

\begin{enumerate}

\item

$\Delta - \Delta ^M=\{\alpha _2,\alpha _4\}$, $M$ is of type
$A_2\times A_2$, $\Sigma _{\mu }=\{\alpha _2,\alpha _4,\alpha
_2+\alpha _4,\alpha _2+2\alpha _4,\alpha _2+3\alpha _4,
2\alpha _2+3\alpha _4\}$ is of type $G_2$, $\alpha _2$ is the
long root.  As $M_{\alpha _2}$ and $M_{\alpha _4}$ are both of
$A$-type, $(\omega ^{M_{\alpha _2}}_{\alpha _2},\omega
^{M_{\alpha _2}}_{\alpha _2})=1/2$ and $(\omega ^{M_{\alpha
_4}}_{\alpha _4},\omega ^{M_{\alpha _4}}_{\alpha _4})=3/2$,
$\epsilon _{\alpha _2}$ and $\epsilon _{\alpha _4}1$ are
necessarily $1$. One deduces that the assumptions are
satisfied.

\end{enumerate}

\null $\mathbf E_7:$

Here the relevant cases are:

\input{E7.TpX}

\begin{enumerate}

\item

$\Delta - \Delta ^M=\{\alpha _1,\alpha _3\}$, $M$ is of type
$A_5$, $\Sigma _{\mu }=\{\alpha _1,\alpha _3, \alpha _1+\alpha
_3, \alpha _1+2\alpha _3,\alpha _1+3\alpha _3, 2\alpha
_1+3\alpha _3\}$ is of type $G_2$, $\alpha _1$ is the long
root, $M_{\alpha _1}$ is of type $A_1\times A_5$, $M_{\alpha
_3}$ is of type $D_6$, $(\omega ^{M_{\alpha _1}}_{\alpha
_1},\omega ^{M_{\alpha _1}}_{\alpha _1})=1/2$ and $(\omega
^{M_{\alpha _3}}_{\alpha _3},\omega ^{M_{\alpha _3}}_{\alpha
_3})=3/2$, $\epsilon _{\alpha _1}$ and $\epsilon _{\alpha _3}$
are always $1$. One deduces that the assumptions are
satisfied.

\item

$\Delta - \Delta ^M=\{\alpha _1,\alpha _6\}$, $M$ is of type
$D_4\times A_1$, $\Sigma _{\mu }=\{\alpha _1,\alpha _6, \alpha
_1+\alpha _6, \alpha _1+2\alpha _6\}$ is of type $B_2$,
$\alpha _1$ is the long root, $M_{\alpha _1}$ and $M_{\alpha
_6}$ are both of type $D_5$, $(\omega ^{M_{\alpha _1}}_{\alpha
_1},\omega ^{M_{\alpha _1}}_{\alpha _1})=1$ and $(\omega
^{M_{\alpha _6}}_{\alpha _6},\omega ^{M_{\alpha _6}}_{\alpha
_6})=2$, $\epsilon _{\alpha _1}$ is always $1$ and $\epsilon
_{\alpha _6}$ can be $1$ or $2$. One deduces that the
assumptions are satisfied.

\item

$\Delta - \Delta ^M=\{\alpha _4,\alpha _6\}$, $M$ is of type
$A_2\times A_1\times A_1\times A_1$, $\Sigma _{\mu }=\{\alpha
_4,\alpha _6, \alpha _4+\alpha _6, 2\alpha _4+\alpha _6,
3\alpha _4+\alpha _6, 3\alpha _4+2\alpha _6\}$ is of type
$G_2$, $\alpha _6$ is the long root, $M_{\alpha _4}$ is of
type $D_5$, $M_{\alpha _6}$ is of type $A_2\times A_1\times
A_3$, $(\omega ^{M_{\alpha _4}}_{\alpha _4},\omega ^{M_{\alpha
_4}}_{\alpha _4})=3$ and $(\omega ^{M_{\alpha _6}}_{\alpha
_6},\omega ^{M_{\alpha _6}}_{\alpha _6})=1$, $\epsilon
_{\alpha _6}$ is always $1$ and it follows from lemma 6.1 that
$\epsilon _{\alpha _4}$ is always $1$, too. One deduces that
the assumptions are satisfied.

\end{enumerate}

$\mathbf E_8:$

The relevant cases are:

\input{E8.TpX}

\begin{enumerate}

\item

$\Delta - \Delta ^M=\{\alpha _1,\alpha _5\}$, $M$ is of type
$A_3\times A_3$, $\Sigma _{\mu }=\{\alpha _1,\alpha _1+\alpha
_5, \alpha _1+2\alpha _5, \alpha _1+3\alpha _5\}$ is of type
$B_2$, $\alpha _1+\alpha _5$ is the long root, $M_{\alpha _5}$
is of type $D_7$, $M_{\alpha _1+\alpha _5}$ is of type $A_7$,
$(\omega ^{M_{\alpha _5}}_{\alpha _5},\omega ^{M_{\alpha
_5}}_{\alpha _5})=4$ and $(\omega ^{M_{\alpha _1+\alpha
_5}}_{\alpha _1+\alpha _5},\omega ^{M_{\alpha _1+\alpha
_5}}_{\alpha _1+\alpha _5})=2$, and $\epsilon _{\alpha
_1+\alpha _5}$ is always $1$ and $\epsilon _{\alpha _1}$ can
be $1$ or $2$. One deduces that the assumptions are satisfied.

\item

$\Delta - \Delta ^M=\{\alpha _1,\alpha _6\}$, $M$ is of type
$D_4\times A_2$, $\Sigma _{\mu }=\{\alpha _1,\alpha _6, \alpha
_1+\alpha _6, \alpha _1+2\alpha _6,\alpha _1+3\alpha _6,
2\alpha _1+3\alpha _6\}$ is of type $G_2$, $\alpha _1$ is the
long root, $M_{\alpha _1}$ is of type $D_5\times A_2$,
$M_{\alpha _6}$ is of type $D_7$, $(\omega ^{M_{\alpha
_1}}_{\alpha _1},\omega ^{M_{\alpha _1}}_{\alpha _1})=1$ and
$(\omega ^{M_{\alpha _6}}_{\alpha _6},\omega ^{M_{\alpha
_6}}_{\alpha _6})=3$, $\epsilon _{\alpha _1}$ is always $1$
and it follows from lemma 6.1 that $\epsilon _{\alpha _6}$ is
always $1$, too. One deduces that the assumptions are
satisfied.

\item

$\Delta - \Delta ^M=\{\alpha _1,\alpha _8\}$, $M$ is of type
$D_6$, $\Sigma _{\mu }=\{\alpha _1,\alpha _8, \alpha _1+\alpha
_8, 2\alpha _1+\alpha _8\}$ is of type $B_2$, $\alpha _8$ is
the long root, $M_{\alpha _1}$ is of type $E_7$, $M_{\alpha
_8}$ is of type $D_7$, $(\omega ^{M_{\alpha _1}}_{\alpha
_1},\omega ^{M_{\alpha _1}}_{\alpha _1})=2$ and $(\omega
^{M_{\alpha _8}}_{\alpha _8},\omega ^{M_{\alpha _8}}_{\alpha
_8})=1$, and $\epsilon _{\alpha _8}$ is always $1$ and
$\epsilon _{\alpha _1}$ can be $1$ or $2$. One deduces that
the assumptions are satisfied.

\item

$\Delta - \Delta ^M=\{\alpha _2,\alpha _5\}$, $M$ is of type
$A_3\times A_3$, $\Sigma _{\mu }=\{\alpha _2,\alpha _2+\alpha
_5, \alpha _2+2\alpha _5, 2\alpha _2+3\alpha _5\}$ is of type
$B_2$, $\alpha _5$ is the long root, $M_{\alpha _5}$ is of
type $A_7$, $M_{\alpha _2+\alpha _5}$ is of type $D_7$,
$(\omega ^{M_{\alpha _5}}_{\alpha _5},\omega ^{M_{\alpha
_5}}_{\alpha _5})=2$ and $(\omega ^{M_{\alpha _2+\alpha
_5}}_{\alpha _2+\alpha _5},\omega ^{M_{\alpha _2+\alpha
_5}}_{\alpha _2+\alpha _5})=4$, and $\epsilon _{\alpha _5}$ is
always $1$ and $\epsilon _{\alpha _2+\alpha _5}$ can be $1$ or
$2$. One deduces that the assumptions are satisfied.

\item

$\Delta - \Delta ^M=\{\alpha _4,\alpha _6\}$, $M$ is of type
$A_2\times A_1\times A_1\times A_2$, $\Sigma _{\mu }=\{\alpha
_4,\alpha _4+\alpha _6, 2\alpha _4+\alpha _6, 3\alpha
_4+2\alpha _6\}$ is of type $B_2$, $\alpha _4$ is the long
root, $M_{\alpha _4}$ is of type $D_5$, $M_{\alpha _4+\alpha
_6}$ is of type $E_6$, $(\omega ^{M_{\alpha _4}}_{\alpha
_4},\omega ^{M_{\alpha _4}}_{\alpha _4})=3$ and $(\omega
^{M_{\alpha _4+\alpha _6}}_{\alpha _4+\alpha _6},\omega
^{M_{\alpha _4+\alpha _6}}_{\alpha _4+\alpha _6})=6$,
$\epsilon _{\alpha _4+\alpha _6}$ can be $1$ or $2$, and it
follows from lemma 6.1 that $\epsilon _{\alpha _4}$ is always
$1$. One deduces that the assumptions are satisfied.

\item

$\Delta - \Delta ^M=\{\alpha _4,\alpha _7\}$, $M$ is of type
$A_2\times A_1\times A_2\times A_1$, $\Sigma _{\mu }=\{\alpha
_4,\alpha _4+\alpha _7, 2\alpha _4+\alpha _7, 3\alpha
_4+\alpha _7\}$ is of type $B_2$, $\alpha _4+\alpha _7$ is the
long root, $M_{\alpha _4}$ is of type $E_6$, $M_{\alpha
_4+\alpha _7}$ is of type $D_5\times A_2$, $(\omega
^{M_{\alpha _4}}_{\alpha _4},\omega ^{M_{\alpha _4}}_{\alpha
_4})=6$ and $(\omega ^{M_{\alpha _4+\alpha _7}}_{\alpha
_4+\alpha _7},\omega ^{M_{\alpha _4+\alpha _7}}_{\alpha
_4+\alpha _7})=3$, $\epsilon _{\alpha _4}$ can be $1$ or $2$
and it follows from lemma 6.1 that $\epsilon _{\alpha
_4+\alpha _7}$ is always $1$. One deduces that the assumptions
are satisfied.

\item

$\Delta - \Delta ^M=\{\alpha _7,\alpha _8\}$, $M$ is of type
$E_6$, $\Sigma _{\mu }=\{\alpha _7,\alpha _8, \alpha _7+\alpha
_8, 2\alpha _7+\alpha _8,3\alpha _7+\alpha _8, 3\alpha
_7+2\alpha _8\}$ is of type $G_2$, $\alpha _8$ is the long
root, $M_{\alpha _7}$ is of type $E_7$, $M_{\alpha _8}$ is of
type $E_6\times A_1$, $(\omega ^{M_{\alpha _7}}_{\alpha
_7},\omega ^{M_{\alpha _7}}_{\alpha _7})=3/2$ and $(\omega
^{M_{\alpha _8}}_{\alpha _8},\omega ^{M_{\alpha _8}}_{\alpha
_8})=1/2$, $\epsilon _{\alpha _8}$ is always $1$ and $\epsilon
_{\alpha _7}$ may be $1$ or $2$. One deduces that the
assumptions are satisfied.
\end{enumerate}

$\mathbf F_4:$

The relevant cases are

\input{F4.TpX}

\begin{enumerate}

\item

$\Delta ^M=\Delta -\{\alpha _1,\alpha _2\}$, $M$ is of type
$A_2$,  $\Sigma _{\mu }=\{\alpha _1,\alpha _2, \alpha
_1+\alpha _2, \alpha _1+2\alpha _2,\alpha _1+3\alpha _2,
2\alpha _1+3\alpha _2\}$ is of type $G_2$, $\alpha _1$ is the
long root, $M_{\alpha _1}$ is of type $A_1\times A_2$,
$M_{\alpha _2}$ is of type $C_3$, $(\omega ^{M_{\alpha
_1}}_{\alpha _1},\omega ^{M_{\alpha _1}}_{\alpha _1})=1/2$ and
$(\omega ^{M_{\alpha _2}}_{\alpha _2},\omega ^{M_{\alpha
_2}}_{\alpha _2})=3/2$, $\alpha _1$ and $\alpha _2$ have both
the same length, $\epsilon _{\alpha _1}$ is always $1$ and it
follows from lemma 6.1 that $\epsilon _{\alpha _2}$ is always
$1$, too. One deduces that the assumptions are satisfied.
\item

$\Delta ^M=\Delta -\{\alpha _1,\alpha _4\}$, $M$ is of type
$B_2$, $\Sigma _{\mu }=\{\alpha _1,\alpha _4, \alpha _1+\alpha
_4, \alpha _1+2\alpha _4\}$ is of type $B_2$, $\alpha _1$ is
the long root, $M_{\alpha _1}$ is of type $B_3$, $M_{\alpha
_4}$ is of type $C_3$, $(\omega ^{M_{\alpha _1}}_{\alpha
_1},\omega ^{M_{\alpha _1}}_{\alpha _1})=1$ and $(\omega
^{M_{\alpha _4}}_{\alpha _4},\omega ^{M_{\alpha _4}}_{\alpha
_4})=1/2$, $(\alpha _1,\alpha _1)=2$, $(\alpha _4,\alpha
_4)=1$, $\epsilon _{\alpha _4}$ is always $1$ and $\epsilon
_{\alpha _1}$ may be $1$ or $2$. One deduces that the
assumptions are satisfied.

\item

$\Delta ^M=\Delta -\{\alpha _3,\alpha _4\}$, $M$ is of type
$A_2$, $\Sigma _{\mu }=\{\alpha _3,\alpha _4, \alpha _3+\alpha
_4, 2\alpha _3+\alpha _4,3\alpha _3+\alpha _4, 3\alpha
_3+2\alpha _4\}$ is of type $G_2$, $\alpha _4$ is the long
root, $M_{\alpha _3}$ is of type $B_3$, $M_{\alpha _4}$ is of
type $A_2\times A_1$, $(\omega ^{M_{\alpha _3}}_{\alpha
_3},\omega ^{M_{\alpha _3}}_{\alpha _3})=3/4$ and $(\omega
^{M_{\alpha _4}}_{\alpha _4},\omega ^{M_{\alpha _4}}_{\alpha
_4})=1/4$, $\alpha _3$ and $\alpha _4$ have both the same
lenght, $\epsilon _{\alpha _3}$ and $\epsilon _{\alpha _4}$
are always $1$. One deduces that the assumptions are
satisfied.

\end{enumerate}

\subsection{The non-split quasi-split cases:}
Here the absolute root system differs from the $F$-root system.
The question of self-conjugacy can be dealt with the $F$-root
system. For the formula which relates $\epsilon _{\beta }$ and
$\epsilon _{\overline{\beta }}$, one has now to use $\ti{\beta }$,
which is a multiple of $\omega _{\beta }^{M_{\beta }}$ by a
nonzero scalar. This scalar is determined by the relation between
the restrictions of $\beta ^{\vee }$ and $\underline{\beta }^{\vee
}$. Remark that all the absolute root systems below are simply
laced, so that the absolute roots have in each case the same
length. We will also use the fact that the $\epsilon _{\alpha }$
are invariant by restriction of scalars.

\null $\mathbf \ ^2A_{2n-1}:$

\input{2Anodd.TpX}

This absolute root system corresponds to quasi-split groups which
split over a quadratic extension $F'$ of $F$. The $F$-root system
is of type $C_n$. Hence we have the same relevant cases as
discussed in the split $C_n$ case. We will denote by $\ti{A}_i$
the type of a quasi-split group which is the restriction of
scalars with respect to $F'/F$ of a split group of type $A_i$.

\begin{enumerate}

\item

$\Delta -\Delta^M=\{\alpha _i,\alpha _j\}$, $1\leq i<j\leq
n-1$, $M$ is of type $\tilde{A}_{i-1} \times \tilde{A}_{j-i-1}
\times\ ^2A_{2(n-j)-1}$, $\Sigma _{red}(T_M)=\{\alpha
_i,\alpha _n,\alpha _i+\alpha _n,\alpha _i+2\alpha _j\}$ is of
type $B_2$, $\alpha _i$ is the long root. If $\Sigma _{\mu }$
is properly contained in $\Sigma (T_M)$, $M_{\alpha _i}$ is of
type $\tilde{A}_{j-1}\times\ ^2A_{2(n-j)-1}$ and $M_{\alpha
_j}$ of type $\tilde{A}_{i-1}\times {}^2A_{2(n-i)-1}$. In
order of $\Sigma _{\mu }$ to be of type $B_2$, $M$ must be
self-conjugate in $M_{\alpha _i}$, which means $j=2i$. Then
$(\ti{\alpha _i},\ti{\alpha _i})=j/2$ and $(\ti{\alpha _j},
\ti{\alpha _j})=j$. As the $\epsilon_{\beta}$ are invariant
for restriction of scalars, we have always
$\epsilon_{\alpha_i}=1$, $\epsilon _{\alpha _j}$ may be $1$ or
$2$. One deduces that our assumptions are satisfied.

\item

$\Delta -\Delta^M=\{\alpha _i,\alpha _n\}$, $1\leq i<n$, $M$
is of type $\tilde{A}_{i-1}\times \tilde{A}_{n-i-1}$, $\Sigma
_{red}(T_M)=\{\alpha _i,\alpha _n,\alpha _i+\alpha _n,2\alpha
_i+\alpha _n\}$ is of type $B_2$, $\alpha _n$ is the long
root. In order of $\Sigma _{\mu }$ to be of type $B_2$, $M$
must be self-conjugate in $M_{\alpha _i}$, which means $n=2i$.
Then $(\ti{\alpha _i},\ti{\alpha _i})=n/2$, $(\ti{\alpha
_n},\ti{\alpha _n})=n/4$, $\epsilon _{\alpha _i}$ is always
$1$ (as in the previous case) and $\epsilon _{\alpha _n}=1$ by
\cite[diagram ${}^2A_{2k-1}-2$]{Sh2}. One deduces that the
assumptions are satisfied.
\end{enumerate}

{\bf ${}^2A_{2n}$:}

\input{2Aneven.TpX}

This absolute root system correponds to $F$-groups which split
over a quadratic extension $F'/F$. The reduced $F$-roots system is
of type $B_n$. Hence we have the same relevant cases as discussed
in the split $B_n$ case.

\begin{enumerate}

\item

$\Delta -\Delta^M=\{\alpha _i,\alpha _j\}$, $1\leq i<j\leq
n-1$, $M$ is of type $\tilde{A}_{i-1}\times
\tilde{A}_{j-i-1}\times {}^2A_{2(n-j)}$, $\Sigma
_{red}(T_M)=\{\alpha _i,\alpha _j,\alpha _i+\alpha _j,\alpha
_i+2\alpha _j\}$ is of type $B_2$, $\alpha _i$ is the long
root, the relevant first factor of $M_{\alpha _i}$ is of type
$\tilde{A}_{j-1}$, the relevant second factor of $M_{\alpha
_j}$ is of type ${}^2A_{2(n-i)}$. In order of $\Sigma _{\mu }$
to be of type $B_2$, $M$ must be self-conjugate in $M_{\alpha
_i}$, which means $j=2i$. Then $(\ti{\alpha _i},\ti{\alpha
_i})=j/2$, $(\ti{\alpha _j},\ti{\alpha _j})=j$,
$\epsilon_{\alpha_i}=1$ as above and, by \cite[diagram
${}^2A_{2k-1}-1,4$]{Sh2}, $\epsilon _{\alpha _j}$ may be $1$
or $2$. One deduces that our assumptions are satisfied.

\item

$\Delta -\Delta^M=\{\alpha _i,\alpha _n\}$, $1\leq i<n$, $M$
is of type $\tilde{A}_{i-1}\times \tilde{A}_{n-i-1}$, $\Sigma
_{red}(T_M)=\{\alpha _i,\alpha _n,\alpha _i+\alpha _n,\alpha
_i+2\alpha _n\}$ is of type $B_2$, $\alpha _i$ is the long
root. In order of $\Sigma _{\mu }$ to be of type $B_2$, $M$
must be self-conjugate in $M_{\alpha _i}$, which means $n=2i$.
Then $(\ti{\alpha _i},\ti{\alpha _i})=n/2$, $(\ti{\alpha
_n},\ti{\alpha _n})=n$, $\epsilon _{\alpha _i}=1$ (as in the
previous case) and $\epsilon _{\alpha _n}$ can be $1$ or $2$
by \cite[diagram ${}^2A_{2k-1}-3$]{Sh2}. One deduces that our
assumptions are satisfied.

\end{enumerate}

{\bf ${}^2D_{n+1}$:}

\input{2Dn.TpX}

This absolute root system correponds to $F$-groups which split
over a quadratic extension $F'/F$. The reduced $F$-roots system is
of type $B_n$. Hence we have the same relevant cases as discussed
in the split $B_n$ case.

\begin{enumerate}
\item

$\Delta -\Delta^M=\{\alpha _i,\alpha _j\}$, $1\leq i<j\leq-1$,
$M$ is of type $A_{i-1}\times A_{j-i-1}\times {}^2D_{n-j+1}$,
$\Sigma _{red}(T_M)=\{\alpha _i,\alpha _j,\alpha _i+\alpha
_j,\alpha _i+2\alpha _j\}$ is of type $B_2$, $\alpha _i$ is
the long root, the relevant first factor of $M_{\alpha _i}$ is
of type $A_{j-1}$, the relevant second factor of $M_{\alpha
_j}$ is of type ${}^2D_{n-i+1}$. In order of $\Sigma _{\mu }$
to be of type $B_2$, $M$ must be self-conjugate in $M_{\alpha
_i}$, which means $j=2i$. Then $(\ti{\alpha _i},\ti{\alpha
_i})=j/2$, $(\ti{\alpha _j},\ti{\alpha _j})=j$, clearly
$\epsilon_{\alpha_i}=1$ and, by \cite[diagram
${}^2D_{n}-1,2$]{Sh2}, $\epsilon _{\alpha _j}$ may be $1$ or
$2$. One deduces that our assumptions are satisfied.

\item

$\Delta -\Delta^M=\{\alpha _i,\alpha _n\}$, $1\leq i<n$, $M$
is of type $A_{i-1}\times A_{n-i-1}$, $\Sigma
_{red}(T_M)=\{\alpha _i,\alpha _n,\alpha _i+\alpha _n,\alpha
_i+2\alpha _n\}$ is of type $B_2$, $\alpha _i$ is the long
root. In order of $\Sigma _{\mu }$ to be of type $B_2$, $M$
must be self-conjugate in $M_{\alpha _i}$, which means $n=2i$.
Then $(\ti{\alpha _i},\ti{\alpha _i})=n/2$, $(\ti{\alpha
_n},\ti{\alpha _n})=n$, $\epsilon _{\alpha _i}=1$ (as in the
previous case) and $\epsilon _{\alpha _n}$ may be $1$ or $2$
by \cite[diagram ${}^2D_{n}-3$]{Sh2}.  One deduces that our
assumptions are satisfied.

\end{enumerate}

\null {\bf ${}^3D_{4}$ and ${}^6D_4$:}

\input{3D4.TpX}

These are the two quasi-split triality $D_4$ groups. The group
$^3D_4$ splits over a (cyclic) extension of degree 3 and the group
$^6D_4$ over a Galois extension of degree 6 with Galois group
$S_3$. So, in both cases the absolute root system is the same,
only the action of the Galois group differs. The $F$-root system
is in both cases of type $G_2$, which is already of rank 2. So the
only relevant case is, when $\Sigma _{\mu }$ equals the $F$-root
system. Denote by $\alpha _1$ the short root and by $\alpha _2$
the long root. As $M_{\alpha _1}$ is of type $A_1$, one has always
$\epsilon _{\alpha _1}=1$. The group $M_{\alpha _2}$ is of type
$\ti{A}_1$, which means that the root system of its $L$-group is
the union of three root systems of type $A_1$ with a transitive
action of the Galois group. One deduces that $\epsilon _{\alpha
_2}$ is always $1$, too. As $(\ti{\alpha _1},\ti{\alpha _1})=2$
and $(\ti{\alpha _2},\ti{\alpha _2})=2/3$, our assumptions are
satisfied.

\null {\bf ${}^2E_{6}$:}

\input{2E6.TpX}

The two quasi-split cases of ${}^2E_6$ type (one has an unramified
quadratic extension as ``splitting field'', the other a ramified
extension of degree $2$) give rise to a relative Dynkin diagram of
type $F_4$ (which dictates the analysis of the relevant cases). In
these cases the analysis is exactly the same. We denote by
$F^\prime$ the splitting field (a quadratic extension of $F$).

\begin{enumerate}

\item

$\Delta ^M=\Delta -\{\alpha _1,\alpha _2\}$, $M$ is of type
$\tilde{A}_2$,  $\Sigma _{\mu }=\{\alpha _1,\alpha _2, \alpha
_1+\alpha _2, \alpha _1+2\alpha _2,\alpha _1+3\alpha _2,
2\alpha _1+3\alpha _2\}$ is of type $G_2$, $\alpha _1$ is the
long root, $M_{\alpha _1}$ is of type $A_1\times \tilde{A}_2$,
$M_{\alpha _2}$ is of type ${}^2A_5$, $(\ti{\alpha _1},
\ti{\alpha _1})=1/2$ and $(\ti{\alpha _2},\ti{\alpha
_2})=3/2$, $\epsilon _{\alpha _1}$ is always $1$ and $\epsilon
_{\alpha _2}$ may be $1$ or $2$ (by \cite[diagram
$^2E6-1$]{Sh2}), and it follows from lemma 6.1 that $\epsilon
_{\alpha _2}$ is always $1$. \rm

\item

$\Delta ^M=\Delta -\{\alpha _1,\alpha _4\}$, $M$ is of type
${}^2A_3$, $\Sigma _{\mu }=\{\alpha _1,\alpha _4, \alpha
_1+\alpha _4, \alpha _1+2\alpha _4\}$ is of type $B_2$,
$\alpha _1$ is the long root, $M_{\alpha _1}$ is of type
${}^2D_4$, $M_{\alpha_4}$ is of type ${}^2A_5$, $(\ti{\alpha
_1},\ti{\alpha _1})=1$ and $(\ti{\alpha _4},\ti{\alpha
_4})=2$, $\epsilon _{\alpha _1}$ is always $1$, and $\epsilon
_{\alpha _4}$ may be $1$ or $2$ by \cite[Diagram
${}^2A_5-1$]{Sh2}. One deduces that the assumptions are
satisfied.

\item

$\Delta ^M=\Delta -\{\alpha _3,\alpha _4\}$, $M$ is of type
$A_2$, $\Sigma _{\mu }=\{\alpha _3,\alpha _4, \alpha _3+\alpha
_4, 2\alpha _3+\alpha _4,3\alpha _3+\alpha _4, 3\alpha
_3+2\alpha _4\}$ is of type $G_2$, $\alpha _4$ is the long
root, $M_{\alpha _3}$ is of type ${}^2D_4$, $M_{\alpha _4}$ is
of type $A_2\times \tilde{A}_1$, $(\ti{\alpha _3},\ti{\alpha
_3})=3$ and $(\ti{\alpha _4},\ti{\alpha _4})=1$, $\epsilon
_{\alpha _4}$ is always $1$ and $\epsilon _{\alpha _3}$ may be
$1$ or $2$ (for the first, use \cite[Diagram ${}^2D_4$]{Sh2},
and it follows from lemma 6.1 that $\epsilon _{\alpha _4}$ is
always $1$. \rm
\end{enumerate}

\end{document}